\documentclass{amsart}
\usepackage{amsmath,xcolor,amssymb}
\usepackage[english]{babel}
\usepackage{hyperref}
\usepackage{tikz}
\usepackage{tikz-cd}

\newcommand{\ZZ}{\mathbb{Z}}

\newcommand{\cH}{\mathcal{H}}

\newcommand{\cM}{\mathcal{M}}

\newcommand{\cC}{\mathcal{C}}
\newcommand{\cD}{\mathcal{D}}
\newcommand{\merge}{\xrightarrow{\mathrm{merge}}}

\newcommand{\sgn}{\mathrm{sgn}}
\newcommand{\id}{\mathrm{id}}
\newcommand{\ten}{\otimes}

\newcommand{\ra}{\rightarrow}
\newcommand{\co}{\colon}

\newtheorem{theorem}{Theorem}[section]
\newtheorem*{theorem*}{Main Theorem}
\newtheorem{proposition}[theorem]{Proposition}
\newtheorem{corollary}[theorem]{Corollary}

\theoremstyle{definition}

\theoremstyle{definition}
\newtheorem{definition}[theorem]{Definition}
\newtheorem{remark}[theorem]{Remark}

\title[Sign-reversing involutions and the matroid-minor Hopf algebra]{Matroid-minor Hopf algebra: a cancellation-free antipode formula and other applications of sign-reversing involutions}
\author{Eric Bucher}
\address{Department of Mathematics,
Xavier University, Cincinnati, Ohio 45207, USA}
\email{buchere1@xavier.edu}

\author{Chris Eppolito}
\address{Department of Mathematical Sciences,
Binghamton University, Binghamton, NY 13902, USA}
\email{eppolito@math.binghamton.edu}

\author{Jaiung Jun}
\address{Department of Mathematics,
University of Iowa, Iowa City, IA, 52242, USA}
\email{jujun0915@gmail.com}

\author{Jacob P. Matherne}
\address{School of Mathematics,
Institute for Advanced Study, Princeton, NJ, 08540, USA}
\email{matherne@math.ias.edu}

\thanks{E.B. received support from AMS-Simons travel grant.  J.J. received support from AMS-Simons travel grant. Part of this work was done while J.J. was working at Binghamton university. J.M. received support from NSF Grant No. DMS-1638352 and the Association of Members of the Institute for Advanced Study (AMIAS)}



\begin{document}

\begin{abstract}
In this paper, we give a cancellation-free antipode formula for the matroid-minor Hopf algebra. We then explore applications of this formula. For example, the cancellation-free formula expresses the antipode of uniform matroids as a sum over certain ordered set partitions.
We also prove that all matroids over any hyperfield (in the sense of Baker and Bowler) have cancellation-free antipode formulas; furthermore, the cancellations in the antipode are independent of the hyperfield structure and only depend on the underlying matroid.
\end{abstract}

\maketitle

\section{Introduction}
\label{sec:intro}

Hopf algebras naturally appear in combinatorics in the following way: one constructs a Hopf algebra whose generators are canonically parametrized by certain combinatorial objects of interest, for instance, graphs, posets, or symmetric functions. Then the Hopf algebra structures encode basic operations of combinatorial objects, such as direct sums, restrictions, or contractions. One of the main motivations for studying these types of Hopf algebras is to obtain combinatorial results by appealing to purely algebraic properties of Hopf algebras (see \cite{gr14}).  

When basic operations of combinatorial objects provide only bialgebra structures, to obtain Hopf algebra structures, one may appeal to Takeuchi's celebrated result \cite{t71}, which states that every graded, connected bialgebra
is a Hopf algebra with an explicit antipode formula. This formula, however, usually contains a large number of cancellations so is not optimal for producing combinatorial identities among the elements of the Hopf algebra.  

Recently, a great deal of research has been devoted to developing cancellation-free antipode formulas for combinatorial Hopf algebras.
For example, cancellation-free antipode formulas have been obtained for the incidence Hopf algebra on graphs by Humpert and Martin \cite{hm12}, for $K$-theoretic analogs of various symmetric function Hopf algebras (introduced by Lam and Pylyavskyy \cite{lp07}) by Patrias \cite{p16}, for the Hopf algebra of simplicial complexes by Benedetti, Hallam, and Macechek \cite{bhm16}, and for various Hopf algebras embeddable into the Hopf monoid by Benedetti and Bergeron \cite{bb16}.

In a seminal paper by Benedetti and Sagan, cancellation-free antipode formulas were given for nine combinatorial Hopf algebras \cite{bs17}.  What is particularly novel about their approach is that all of their cancellation-free formulas are obtained via the same general technique---for each of their combinatorial Hopf algebras, they introduce a sign-reversing involution on the set indexing the terms in Takeuchi's formula, and they classify the fixed points of this involution.  Summing over these fixed points yields a cancellation-free formula (see Section~\ref{sec:examplesignreverse} for an elementary example of a sign-reversing involution). Examples of combinatorial Hopf algebras that yield to their technique are the shuffle Hopf algebra, the incidence Hopf algebra for graphs, and one of the Lam--Pylyavskyy symmetric function Hopf algebras (thereby recovering the formula of Patrias). The number of recent results producing cancellation-free antipode formulas via sign-reversing involutions suggests that potentially there is a general way to define the involution for a Hopf algebra whose generators are obtained from combinatorial objects. The results presented in this paper hope to further the evidence that such an involution can be constructed. 

The matroid-minor Hopf algebra was introduced by Schmitt in 1994 \cite{s94}, and a cancellation-free formula for the antipode of uniform matroids was given by the first and last authors in 2016 \cite{bm16} by using a sign-reversing involution.  The current paper finishes what was started in \cite{bm16} and gives a cancellation-free antipode for all matroids (as well as matroids over hyperfields, in Corollary~\ref{coro:hyperfield theorem}).  

\begin{theorem*}[Theorem~\ref{thm:cancellation-free antipode}]
Let $M$ be a matroid with ground set $E$. Choose a total ordering on $E$, call it $<$. Then
\[
S(M)=\sum_{ \text{fixed points $\pi$ of } \iota_<} \ \sgn(\pi) \ \bigoplus_{i=1}^{j} U_{\delta_i^M-\sigma_i^M}^{| \delta_i^M-\sigma_i^M |} \oplus (M |_{\delta_i^M}) / (\delta_i^M - \sigma_i^M).
\]
\end{theorem*}

In the formula above, $\iota_<$ is a certain sign-reversing involution, defined from $<$, on the set of ordered set partitions of the ground set $E$ of the matroid $M$ (see Section~\ref{sec:takeuichi}), the matroid $U^{|P|}_P$ is the uniform matroid of rank $|P|$ on ground set $P$, the matroid $(M|_P)/Q$ is the matroid formed by first restricting M to $P$ then contracting $Q$ from $M|_P$, and the various $\delta$ and $\sigma$ are certain subsets of $E$ which carry information about circuits in the matroid $M$ (see Section~\ref{sec:deltasets}).

Between the time of \cite{bm16} and the current paper, a cancellation-free antipode formula for the matroid-minor Hopf algebra was obtained by Aguiar and Ardila \cite{aa17} as an application of a far-reaching approach which involves embedding the matroid-minor Hopf algebra into the Hopf algebra of generalized permutohedra.  Even though a cancellation-free formula already exists, we hope to convince the reader that further study of sign-reversing involutions in the theory of combinatorial Hopf algebras is warranted:

\begin{itemize}
\item Our formula answers part of a question by Benedetti--Sagan \cite[Section 11, Question 1]{bs17} by adding another involved example to the growing list of graded connected Hopf algebras for which cancellation-free antipodes can be obtained via sign-reversing involutions. This suggests that there might be a general procedure one could follow in the case of any combinatorial Hopf algebra.

\item Our antipode formula is given purely combinatorially.  While
  \cite{aa17} gives a much more elegant description of the antipode
  via an embedding into the Hopf algebra of generalized permutohedra,
  our antipode computations work on a ``calculus of partitions'' and
  can be automated using standard computer software.

\item In \cite{baker2017matroids}, Baker and Bowler introduced the notion of \emph{matroids over hyperfields} which unifies various generalizations of matroids including oriented matroids and valuated matroids. Consequently, in \cite{ejs17}, Hopf algebras for matroids over hyperfields are defined. Our method for the antipode formula is robust enough to obtain a cancellation-free antipode formula for the hyperfield case without much efforts. 
\end{itemize}

\subsection*{Acknowledgements}

The authors would like to thank Federico Ardila, Carolina Benedetti, Vic Reiner, and Bruce Sagan for useful conversations while this work was in progress. The authors also thank Binghamton University where the authors met and initiated the project. 

\section{Preliminaries}
\label{sec:prelim}

\subsection{Matroids}

In this paper, we work with the ``circuit'' formulation of a matroid
(as opposed to \cite{bm16} where the first and last authors used
independent sets).

\begin{definition}
  Let $E$ be a finite set, and let $\cC$ be a collection of subsets of
  $E$ satisfying:
  \begin{itemize}
  \item[(a)] $\emptyset \not\in \cC$,
  \item[(b)] if $C_1,C_2 \in \cC$ with $C_1 \subseteq C_2$, then
    $C_1 = C_2$, and
  \item[(c)] (Circuit elimination) if $C_1,C_2 \in \cC$ with $C_1 \not= C_2$ and
    $e \in C_1 \cap C_2$, then there exists $C_3 \in \cC$ with
    $C_3 \subseteq (C_1 \cup C_2) - e$.
\end{itemize}
The pair $M = (E,\cC)$ is a {\em matroid with ground set $E$ and set
  of circuits $\cC$}.  A subset $S\subseteq E$ is \emph{dependent in $M$}
when $S$ contains some circuit of $M$.
\end{definition}

Uniform matroids form an important class of matroids: we write $U^r_n$
for the {\em uniform matroid of rank $r$ on an $n$-element set}---this
is the matroid with ground set $E$ an $n$-element set and set of
circuits $\cC$ all subsets of $E$ with exactly $r+1$ elements.
Throughout, it will be important for us to keep track of various
uniform matroids with different ground sets; when we want to specify
the ground set of a uniform matroid, we write $U_E^r$ for the uniform
matroid $U_{|E|}^r$ with specified ground set $E$.

When defining the operations inside the matroid-minor Hopf algebra we
will need the matroid operations of restriction, contraction, and
direct sum.  To this end, let $M_1 = (E_1, \cC_1)$ and
$M_2 = (E_2, \cC_2)$ be two matroids on disjoint ground sets, and let
$S$ be a subset of $E_1$.  The {\em restriction of $M_1$ to $S$} is
the matroid $M_1|_S = (S, \cD_1)$ where
$\cD_1 = \{C \subseteq S \mid C \in \cC_1\}$.  The {\em contraction of
  $S$ from $M_1$} is the matroid $M_1/S = (E_1-S, \cD)$ where $\cD$ is
the set of minimal non-empty elements of $\{C - S \mid C \in \cC_1\}$.
The {\em direct sum of matroids $M_1$ and $M_2$} is the matroid
$M_1 \oplus M_2 = (E_1 \cup E_2, \cC_1 \cup \cC_2)$.  Note that
$M_1 \oplus M_2 = M_2 \oplus M_1$.

\subsection{Matroid-minor Hopf algebras}
\label{sec:mmh}

For the remainder of the paper, fix a field $\Bbbk$.  A
$\Bbbk$-bialgebra is simultaneously a $\Bbbk$-algebra and a
$\Bbbk$-coalgebra together with some compatibility of the
multiplication morphism $\mu\co \cH \ten \cH \ra \cH$ and unit
morphism $\eta\co \Bbbk \ra \cH$, with the comultiplication morphism
$\Delta\co \cH \ra \cH \ten \cH$ and counit morphism
$\epsilon\co \cH \ra \Bbbk$.  A $\Bbbk$-bialgebra $\cH$ is {\em
  graded} if $\cH = \bigoplus_{\ell \in \ZZ_{\ge 0}} \cH_\ell$, and
the maps $\mu, \eta, \Delta,$ and $\epsilon$ are graded $\Bbbk$-linear
maps.  Additionally, $\cH$ is called {\em connected} if
$\cH_0 = \Bbbk$.

A Hopf algebra $\cH$ over $\Bbbk$ is a $\Bbbk$-bialgebra together with
one additional $\Bbbk$-linear map---the antipode map
$S\co \cH \ra \cH$ also satisfies some compatibility with the
$\Bbbk$-linear maps $\mu, \eta, \Delta,$ and $\epsilon$.  The
celebrated 1971 result by Takeuchi asserts that any graded, connected
$\Bbbk$-bialgebra is a Hopf algebra with an explicit antipode.  We
refer the reader to \cite{sweedler1969hopf} for an introduction to
Hopf algebras.

\begin{theorem}[\cite{t71}]\label{thm:t}
A graded, connected $\Bbbk$-bialgebra $\cH$ is a Hopf algebra, and it has a unique antipode $S$ given by
\begin{equation}\label{eq:tak}
S = \sum_{i \in \ZZ_{\ge 0}} (-1)^i \mu^{i-1} \circ \mathrm{pr}^{\otimes i} \circ \Delta^{i-1},
\end{equation}
where $\mu^{-1} := \eta$, $\Delta^{-1} := \epsilon$, and $\mathrm{pr}\co \cH \rightarrow \cH$ is the projection map defined by linearly extending the map
\[
\mathrm{pr}|_{\cH_{\ell}} = \begin{cases} 0 & \text{if }\ell = 0,\\ \id & \text{if } \ell \ge 1.\end{cases}
\]
\end{theorem}

Throughout this paper, we will focus on a particular Hopf algebra,
called the matroid-minor Hopf algebra, which we will now define.  Let
$\cM$ be any class of matroids which is closed under taking direct
sums and minors, and let $\widetilde{\cM}$ be the set
of all isomorphism classes of matroids in $\cM$.  Then
$\Bbbk \widetilde{\cM}$ is a Hopf algebra \cite{s94} with
$\Bbbk$-linear maps
\[\begin{array}{rcl}
\Bbbk \widetilde{\cM} \ten \Bbbk \widetilde{\cM} & \stackrel{\mu}{\longrightarrow} & \Bbbk \widetilde{\cM} \\
(M,N) & \longmapsto & M \oplus N, \\
& & \\
\Bbbk \widetilde{\cM} & \stackrel{\Delta}{\longrightarrow} & \Bbbk \widetilde{\cM} \ten \Bbbk \widetilde{\cM} \\
M & \longmapsto & \displaystyle \sum_{A \subseteq E} M|_A \ten M/A, \\
& & \\
\Bbbk & \stackrel{\eta}{\longrightarrow} & \Bbbk \widetilde{\cM} \\
1_{\Bbbk} & \longmapsto & U^0_0, \\
& & \\
\Bbbk \widetilde{\cM} & \stackrel{\epsilon}{\longrightarrow} & \Bbbk \\
M & \longmapsto & \left\{\begin{array}{ll}1_{\Bbbk} & \mathrm{if}\ E = \emptyset,\\ 0 & \mathrm{else}.\end{array}\right.
\end{array}\]
The matroid-minor Hopf algebra $\Bbbk\widetilde{\cM}$ is indeed a graded, connected $\Bbbk$-bialgebra, so its antipode is given by Theorem~\ref{thm:t}.
In Section~\ref{sec:rewrite}, we will rewrite Takeuchi's antipode formula \eqref{eq:tak} more explicitly for the matroid-minor Hopf algebra $\Bbbk\widetilde{\cM}$.

\subsection{Sign-reversing involutions}\label{sec:signreversinginvs}

Sign-reversing involutions are ubiquitous in combinatorics and are a well-known way of removing cancellations from a formula.  We give preliminaries on sign-reversing involutions, and we give an elementary example in Section~\ref{sec:examplesignreverse}.

\begin{definition}
Let $A$ be a finite set equipped with a sign function; that is, a function $\sgn\co A \ra \{\pm1\}$.  A {\em sign-reversing involution} $\iota\co A \ra A$  is an involution such that for every $a \in A$
\[
\sgn(\iota(a)) = -\sgn(a).
\]
\end{definition}

Now fix a sign-reversing involution $\iota$ of a finite set $A$.  Let $F \subset A$ be the set of fixed points of $\iota$; that is,
\[
F := \{a \in A \mid \iota(a) = a\}.
\]  
A formula of the form
\[
\sum_{a \in A} \sgn(a)
\]
may have many cancellations, but because $\iota$ is sign-reversing, it follows that
\[
\sum_{a \in A} \sgn(a) = \sum_{a \in F} \sgn(a),
\]
and the right-hand side is cancellation free.

\subsection{An example of a sign-reversing involution}\label{sec:examplesignreverse}
We would like to present a small example to illustrate how one might
use sign-reversing involutions to prove combinatorial identities. We
will use a sign-reversing involution to compute the following
summation:
\[
\sum_{k \geq 0} (-1)^k {\binom{n}{k}} .
\]
For this example we think of $\binom{n}{k}$ as enumerating the ways of partitioning $[n]$ into two sets $A$ and $B$ with $|A|=k$ and $|B|=n-k$. Then this summation assigns a $1$ or $-1$ for each 2-part partition $(A,B)$ where the sign is given by the parity of the cardinality of $A$. The sum can be rewritten then as,

\[
\sum_{(A,B)} (-1)^{|A|} .
\]
We define an involution $\iota$ on the set of 2-part partitions as
follows:

\[  
\iota (A,B) = 
\begin{cases}
     (A - \{n\},B \cup \{n\})  & \text{ if } n\in A, \\
     (A \cup \{n\},B -\{n\})  & \text{ if } n\in B. 
\end{cases}
\]
This is a sign-reversing involution on the index set of the summation. It has no fixed points, and therefore 

\[
\sum_{k \geq 0} (-1)^k \binom{n}{k} =  \sum_{(A,B)} (-1)^{|A|}  = \sum_{(A,B) \in F} (-1)^{|A|} = 0.
\]

\section{The sign-free Takeuchi term}
\label{sec:takeuichi}

We will see in Section~\ref{sec:rewrite} that Takeuchi's antipode formula \eqref{eq:tak} can be reinterpreted, for the matroid-minor Hopf algebra $\Bbbk\widetilde{\cM}$, as a sum over ordered set partitions of the ground set $E$ (see Section~\ref{sec:rewrite}).  Thus, our technique for proving Theorem~\ref{thm:cancellation-free antipode} involves introducing a sign-reversing involution (see Section~\ref{sec:inv}) on the set of ordered set partitions of the ground set $E$ of a matroid $M$.

\subsection{Takeuchi's result for the matroid-minor Hopf algebra}\label{sec:rewrite}

\begin{definition}
An {\em ordered set partition} $\pi =(\pi_1, \pi_2, \ldots, \pi_k)$ of a finite set $E$ is a tuple of nonempty subsets of $E$ satisfying
\begin{itemize}
\item[(a)] $\pi_i \cap \pi_j = \emptyset$ for all $1 \le i\neq j \le k$, and
\item[(b)] $\cup_{i=1}^k \pi_i = E$.
\end{itemize}
We will write $\pi \vDash E$ to denote that $\pi$ is an ordered set partition of $E$.
\end{definition}

Below we rewrite the antipode formula in \eqref{eq:tak} for the matroid-minor Hopf algebra $\Bbbk\widetilde{\cM}$ defined in Section~\ref{sec:mmh}.

\begin{proposition}\cite[Proposition 3.1]{bm16}\label{prop:matant}
Let $M \in \cM$ be a matroid with ground set $E$.  For the matroid-minor Hopf algebra $\Bbbk\widetilde{\cM}$, Takeuchi's formula \eqref{eq:tak} is equivalent to
\[
S(M) = \sum_{k \ge 0} (-1)^k \sum_{(\pi_1,\ldots,\pi_k)\vDash E} M|_{\pi_1} \oplus (M/\pi_1)|_{\pi_2} \oplus \cdots \oplus (M/\bigcup_{i=1}^{k-2}\pi_i)|_{\pi_{k-1}} \oplus M/\bigcup_{i=1}^{k-1} \pi_i.
\]
\end{proposition}

\begin{remark}
In Proposition \ref{prop:matant}, all matroids that appear in the formula are representatives chosen from their isomorphism classes.
\end{remark}

Given an ordered set partition $\pi = (\pi_1,\ldots,\pi_k) \vDash E$, we write $T(\pi)$ for the corresponding {\em signless Takeuchi term}
\[
T(\pi) := M|_{\pi_1} \oplus (M/\pi_1)|_{\pi_2} \oplus \cdots \oplus (M/\bigcup_{i=1}^{k-2}\pi_i)|_{\pi_{k-1}} \oplus M/\bigcup_{i=1}^{k-1} \pi_i.
\] 
In the next section, we will provide a characterization of the signless term $T(\pi)$.  To do so, we first develop the language of $\delta$-sets.

\subsection{Defining $\delta$-sets}\label{sec:deltasets}

Let $M$ be a matroid with ground set $E$, and let $\pi = (\pi_1, \pi_2, \ldots, \pi_k) \vDash E$.  We define the following for the pair $(M,\pi)$:

\begin{itemize}
\item The {\em $\ell$-th part of $(M,\pi)$} is defined as
\begin{equation} \label{equation: l-part}
\ell(M,\pi) := \min\left\{t\ |\ \bigcup_{1 \le i \le t} \pi_i \text{ is a dependent set}\right\}.
\end{equation}
 \item The {\em $\delta$-function of $(M,\pi)$} is defined as
\[
\delta(M,\pi) := \bigcup_{1 \le i \leq \ell(M,\pi)} \pi_i.
\]
\item The {\em $\sigma$-function of $(M,\pi)$} is defined as
\[
\sigma(M,\pi) := \pi_{\ell(M,\pi)}.
\]
\end{itemize}

Now we are ready to define the $\delta$-sets and $\sigma$-sets for
$(M,\pi)$; intuitively, $\delta$-sets keep track of where new circuits
are introduced in the union of the initial parts of the partition
$\pi$.  Note that when $\pi=(\pi_1,\pi_2,\ldots,\pi_k)$ is an ordered
partition of set $E$, we write $\pi/\bigcup_{i=1}^m\pi_i$ for the
ordered partition $(\pi_{m+1},\ldots,\pi_k)$ of
$E-\left(\bigcup_{i=1}^m\pi_i\right)$.  Formally:

\begin{definition}
  We define the $\delta$-sets iteratively as follows:
\begin{align*}
  \delta_1^M&:=\delta(M,\pi),\\
  \delta_{j+1}^M
            &:=\delta(M / U_j,\pi/U_j)
              \text{ when }\delta_j(M,\pi)\neq\delta(M / U_j,\pi/U_j),
\end{align*}
where $U_j=\bigcup_{i=1}^j \delta_i^M$ for all appropriate $j$.  We
call $(\delta_1^M, \delta_2^M, \dots , \delta_j^M)$ the {\em
  $\delta$-sets of $(M,\pi)$}.
\end{definition}

We also define the $\sigma$-sets for $(M,\pi)$, which are determined by the $\delta$-sets of $(M,\pi)$.

\begin{definition}
We define the $\sigma$-sets iteratively as follows:
\begin{align*}
  \sigma_1^M&:=\sigma(M,\pi),\\
  \sigma_{j+1}^M
            &:=\sigma(M / U_j,\pi/U_j)
              \text{ when }\sigma_j(M,\pi)\neq\sigma(M / U_j,\pi/U_j),
\end{align*}
where $U_j=\bigcup_{i=1}^j \delta_i^M$ for all appropriate $j$.  We
call $(\sigma_1^M, \sigma_2^M, \dots , \sigma_j^M)$ the {\em
  $\sigma$-sets of $(M,\pi)$}.
\end{definition}

\begin{proposition}
\label{prop:signless tak}
The signless Takeuchi term $T(\pi)$ is completely determined by the $\delta$-sets of $(M,\pi)$.  Specifically,
\[
T(\pi) = \bigoplus_{i=1}^{j} U_{\delta_i^M-\sigma_i^M}^{| \delta_i^M-\sigma_i^M |} \oplus (M |_{\delta_i^M}) / (\delta_i^M - \sigma_i^M).
\]
\end{proposition}

\begin{proof}

Computing $T(\pi)$ is independent of the order in which we apply the restrictions and contractions dictated by the partition $\pi$.  We will let the $\delta$-sets prescribe to us the order. The first restriction/contraction we perform is

\[
M|_{\delta_1^M} \oplus M/{\delta_1^M}.
\]

Consider the restriction/contraction that is applied to $M$ dictated by the part of $\pi$ immediately proceeding $\sigma_1^M$.  After applying this operation, we obtain

\begin{align*}
&(M|_{\delta_1^M})|_{\delta_1^M - \sigma_1^M} \oplus (M|_{\delta_1^M})/(\delta_1^M - \sigma_1^M) \\
&= M|_{\delta_1^M - \sigma_1^M} \oplus (M|_{\delta_1^M})/(\delta_1^M - \sigma_1^M)
\end{align*}

The first summand contains no dependent sets by the definition of $\delta_1^M$. Therefore $M|_{\delta_1^M - \sigma_1^M}  = U_{\delta_1^M - \sigma_1^M} ^{|\delta_1^M - \sigma_1^M|}$. Furthermore any additional restrictions/contractions dictated by $\pi$ which affect this summand will also result in the matroid $U_{\delta_1^M - \sigma_1^M} ^{|\delta_1^M - \sigma_1^M|}$. Therefore after applying all of the prescribed restriction/contraction operations up to $\delta_1^M$, we have the term

\[
U_{\delta_1^M - \sigma_1^M} ^{|\delta_1^M - \sigma_1^M|} \oplus (M|_{\delta_1^M})/(\delta_1^M - \sigma_1^M) \oplus M/\delta_1^M.
\]

Now we iterate the process on $M/\delta_1^M$ as if it were its own matroid that did not arise from any operations. The result is the sum

\[
T(\pi) = \bigoplus_{i=1}^{j} U_{\delta_i^M-\sigma_i^M}^{| \delta_i^M-\sigma_i^M |} \oplus (M |_{\delta_i^M}) / (\delta_i^M - \sigma_i^M).
\]
\end{proof}

\section{The involution $\iota_<$}
\label{sec:inv}

In this section, we construct a sign-reversing involution $\iota_<$
with the property $T(\iota_< (\pi))=T(\pi)$.  Let $M$ be a matroid
with ground set $E$, and fix a total ordering $<$ of $E$.  We will
define $\iota_<$ as a ``split-merge process'' on the ordered set
partitions of $E$.  For clarity, we first give a heuristic definition in the next paragraph, and then give the formal definition immediately afterward.

For simplicity, we write $\iota$ for $\iota_<$.\footnote{In general, the involution $\iota_<$ is dependent on the choice of total ordering $<$, and in fact there is a distinct involution for each choice of $<$.  Regardless of which total ordering we use for defining $\iota_<$, the cancellation-free antipode formula obtained in Theorem~\ref{thm:cancellation-free antipode} is the same.} The way $\iota$ is
  applied is that it starts with the first part of the partition and
  attempts to split the part. If it can, it splits the part into two
  new parts and then $\iota$ stops. If it cannot, it will attempt to
  merge the part with its successor, combining them into one part and
  then $\iota$ stops. If it neither split nor merged the part, then the
  process moves on to attempt to split or merge the next part. It will
  do this until it either is applied to a part or passes through the
  entire partition and is never applied.

Now we will give the formal definition of splitting and merging.
\begin{description}
\item[Splitting] If the part $\pi_i$ has more than one element, then $\iota$ will try to split this part.  
\begin{itemize}
\item Let $x$ be the largest element in $\pi_i$ with respect to $<$. If $\pi_i$ splits, replace the part $\pi_i$ with the two parts
\[
\{x\} \ \ \ \ \ \ \ \text{and}\ \ \ \ \ \ \ \pi_i - \{x\}.
\]
Hence the new partition after splitting is \[\iota(\pi) =( \pi_1,\pi_2, \dots , \{x\} , \pi_i -\{x\} , \pi_{i+1}, \dots, \pi_k).\]
\item The map $\iota$ will \textbf{only} split if $\pi_i$ and $\iota(\pi)$ have the same $\delta$-sets and $\sigma$-sets as $\pi$. 

\end{itemize}
\vspace{.3cm}

\item[Merging] If the part $\pi_i$ is a single element, then $\iota$ will try to merge it with the part immediately following it. Let $\pi_i={x}$.

\begin{itemize}
\item Merging two parts results in a new ordered set partition that replaces $\pi_i$ and $\pi_{i+1}$ with the single part $\pi_i \cup \pi_{i+1}$. Hence the new partition after merging is \[\iota(\pi)= (\pi_1,\pi_2,\dots, \pi_{i-1},\pi_i\cup \pi_{i+1},\pi_{i+2},\dots, \pi_k).\]

\item The map $\iota$ will \textbf{only} apply a merge if $x > y $ for all $y \in \pi_{i+1}$, and
\item the $\delta$-sets and $\sigma$-sets of $\iota(\pi)$ are the same as the $\delta$-sets and $\sigma$-sets of $\pi$.
\end{itemize}
\end{description}

We will now show that $\iota$ is a sign-reversing
involution.  In order to prove this, we will use the following
observation multiple times.

\begin{remark}
\label{rem:delta}
The most important aspect of $\iota$ is that it preserves the $\delta$-sets and $\sigma$-sets of the pair $(M,\pi)$.  This property of $\iota$ will play a key role in the ``pairing off'' of ordered set partitions needed to produce a cancellation-free antipode formula.
\end{remark}

This remark along with Proposition \ref{prop:signless tak} gives the following result.

\begin{theorem}
\label{thm:involution}
Let $M$ be a matroid with ground set $E$, and let $\pi \vDash E$. Then $T(\pi)=T(\iota(\pi))$.
\end{theorem}

The goal was to construct a sign-reversing \emph{involution}. We need to show that $\iota$ in fact accomplishes this.

\begin{theorem}
The map $\iota$ is an involution on the set of ordered set partitions of $E$.
\end{theorem}

\begin{proof}
Let $\pi \vDash E$. There are two cases to check:  if $\iota(\pi)$ results in a split, and if $\iota(\pi)$ results in a merge. If $\pi$ is neither split nor merged by $\iota$ then it is a fixed point of $\iota$, and hence we have $\iota^2(\pi)=\pi$.

\vspace{.3cm}

\underline{Case One}: $\iota$ splits $\pi$.

\vspace{.3cm}

We have that $\iota^2(\pi)= \iota(( \pi_1,\pi_2, \dots , \{x\} , \pi_i -\{x\} , \pi_{i+1}, \dots, \pi_k))$. Now let us consider what $\iota$ will do for this new partition. It will not attempt to split or merge $\pi_1,\pi_2, \dots, \pi_{i-1}$ because they were not altered during $\iota(\pi)$ and $\iota$ attempts to split or merge starting from left to right. The first part of $( \pi_1,\pi_2, \dots , \{x\} , \pi_i -\{x\} , \pi_{i+1}, \dots, \pi_k)$ which $\iota$ will attempt to be applied to is the part $\{x\}$. Since $x$ was split from $\pi_i$, we know that $x>y$ for all $y\in\pi_{i}$. This is exactly the property that is required for $\iota$ to merge $\{x\}$ with the part immediately following it. Hence we get

\[\iota^2(\pi)= \iota(( \pi_1,\pi_2, \dots , \{x\} , \pi_i -\{x\} , \pi_{i+1}, \dots, \pi_k))=\pi.\]

\vspace{.3cm}

\underline{Case Two}: $\iota$ merges $\pi$.

\vspace{.3cm}

We have that $\iota^2(\pi)= \iota((\pi_1,\pi_2,\dots, \pi_{i-1},\pi_i\cup \pi_{i+1},\pi_{i+2},\dots, \pi_k))$. By Remark \ref{rem:delta}, we know that both splitting and merging will never affect the $\delta$-sets and $\sigma$-sets. We can also see that $\iota$ will not apply to the parts $\pi_1, \pi_2, \dots$, and $\pi_{i-1}$ or it would have been altered during the initial application of $\iota$. Therefore the first time $\iota$ will attempt to be applied is to the part $\pi_i \cup \pi_{i+1}$. This part will split. Since $\iota(\pi)$ resulted in a merge, we know that $\pi_i=\{x\}$ and $x>y$ for all $y\in \pi_{i+1}$. Therefore

\begin{align*}
\iota^2(\pi)&= \iota((\pi_1,\pi_2,\dots, \pi_{i-1},\pi_i\cup \pi_{i+1},\pi_{i+2},\dots, \pi_k))\\
&=(\pi_1,\pi_2,\dots, \pi_{i-1},\pi_i, \pi_{i+1},\pi_{i+2},\dots, \pi_k)\\
&=\pi.
\end{align*}
In both cases we have that $\iota^2(\pi)=\pi$, as desired.
\end{proof}

Until this point, we considered only the {\em signless Takeuchi term} $T(\pi)$.  Now we take into account signs.

\begin{definition}
Let $\pi=(\pi_1,\pi_2,\dots, \pi_k) \vDash E$, where $E$ is the ground set of a matroid $M$.  Define \[ \sgn(\pi) := (-1)^k.\]  
\end{definition}
This definition exactly produces the sign associated to the signless term $T(\pi)$ in the antipode formula of Proposition~\ref{prop:matant}.

\begin{theorem}
\label{thm:signreversing}
The involution $\iota$ is sign-reversing in that; if $\pi$ is not a fixed point of $\iota$, then $\sgn(\pi)=-\sgn(\iota(\pi))$. 
\end{theorem}

\begin{proof}
Let $\pi$ be a non-fixed point of $\iota$. Then $\iota(\pi)$ either is split or merged by $\iota$. Therefore $\iota(\pi)$ has exactly one more or one fewer part than $\pi$.
\end{proof}

\subsection{Characterizing the fixed points of $\iota$}

Now we will characterize the fixed points of $\iota$.

\begin{theorem}
Let $M$ be a matroid with ground set $E$. The fixed points of $\iota$ are the ordered set partitions $\pi$ of $E$ which satisfy the following:

\begin{itemize}
\item For each $i$ which is not the occurrence of a $\sigma$-set, we have $\pi_i$ consists of a single element, and
\item between each occurrence of a $\sigma$-set, these single element parts are in ascending order. 
\end{itemize}
\end{theorem}

\begin{proof}
First we will show that each partition satisfying these two conditions will be a fixed point. Note that splitting will never occur at a $\sigma$-set because this would alter the overall $\sigma$-sets of the partition. Additionally, the first condition assures that the partition will not be split by $\iota$ at any other parts as splitting does not apply to single element parts.

Now we consider merging. We will never merge a singleton part with a $\sigma$-set as this would change the overall $\sigma$-sets of the partition. Therefore we only have to be concerned with the case where a non-$\sigma$ part merges with the non-$\sigma$ part immediately after it. This does not occur because of the second criterion.

To finish the characterization, we must explain why every fixed point satisfies our criteria. Let $\pi$ be a fixed point of $\iota$. Since $\iota$ would split any non-$\sigma$ part with cardinality greater than one, we know that each part which is not the occurrence of a $\sigma$-set will have a single element in $\pi_i$. Additionally since $\iota$ will never apply a merge, this requires that each of these parts \emph{fail} the merge test, meaning they are in ascending order.
\end{proof}

\begin{theorem}
  \label{thm:cancellation-free antipode}
Let $M$ be a matroid with ground set $E$. Choose a total ordering on $E$, call it $<$. Then

\[
S(M)=\sum_{ \text{fixed points $\pi$ of } \iota_<} \ \sgn(\pi) \ \bigoplus_{i=1}^{j} U_{\delta_i^M-\sigma_i^M}^{| \delta_i^M-\sigma_i^M |} \oplus (M |_{\delta_i^M}) / (\delta_i^M - \sigma_i^M).
\]
\end{theorem}

\begin{proof}
This follows directly from Theorem \ref{thm:signreversing}, Theorem \ref{thm:involution}, and Proposition \ref{prop:signless tak}, and by using the principles outlined in Section~\ref{sec:signreversinginvs} about utilizing \emph{sign-reversing involutions} to produce cancellations in sums. To show this is cancellation free, let $T(\pi)=T(\pi ')$ for two fixed points $\pi$ and $\pi '$. This means that they must have the same $\sigma$-sets, though they could occur in different orders. In other words if $(\sigma_1^M, \sigma_2^M, \dots, \sigma_k^M)$ are the $\sigma$-sets for $(M,\pi)$, then the $\sigma$-sets for $(M,\pi')$ must be a permutation of that $k$-tuple. But then notice that in both tuples there would be exactly the same number of parts where a $\sigma$-set occurs, with exactly the same elements. Since all the elements not occuring in $\sigma$-sets occur as singleton parts, we have that $\sgn(\pi)=\sgn(\pi ')$. Therefore this formula is indeed cancellation free.
\end{proof}

\section{Applications}

In this section, we give some applications of Theorem~\ref{thm:cancellation-free antipode}.  In particular, in Section~\ref{sec:uniform}
we recover a formula of \cite{bm16} by interpreting Theorem~\ref{thm:cancellation-free antipode} for the class of uniform matroids.  Moreover, we refine this result by producing a formula that is also grouping-free.  Section~\ref{sec:hyperfield} shows that the techniques of this paper also yield an analogous cancellation-free formula for all matroids over hyperfields.

\subsection{Uniform matroids}
\label{sec:uniform}

A cancellation-free formula for the antipode of uniform matroids via
the involution method was first described in \cite{bm16}, and Theorem
\ref{thm:cancellation-free antipode} is a natural generalization of
the main result of that paper.

\begin{corollary}
  The antipode of a uniform matroid $U_{n}^r$ is given by
  \[
    S(U_n^r)=\sum_{I,J}(-1)^{n-|J|+1}U_I^{|I|}\oplus U_J^{r-|I|},
  \]
  where $I,J$ range over all pairs of subsets of the ground set $E$
  such that
  \begin{itemize}
  \item $I$ and $J$ are disjoint,
  \item $|I|<r$,
  \item $|I|+|J|\geq r$, and
  \item if $|I|+|J|=r$, then $J=\{x\}$ is a singleton and $\max(I)<x$.
  \end{itemize}
\end{corollary}

\begin{proof}
Consider the antipode formula in Theorem \ref{thm:cancellation-free antipode}. Since every circuit in $U_{n}^r$ has the same cardinality we can see that there is exactly one $\delta$-set which is not trivial, $\delta_1^M$. Therefore,
\[S(U_{n}^r) = \sum_{ \text{fixed points $\pi$ of } \iota_<} \ \sgn(\pi) \ U_{\delta_1^M-\sigma_1^M}^{| \delta_1^M-\sigma_1^M |} \oplus (M |_{\delta_1^M}) / (\delta_1^M - \sigma_1^M).
\]

Direct computation of the $\sgn(\pi)$ and the second term in the internal summand produces the desired formula. The indexing set in the summand can be found by carefully characterizing the fixed points of $\iota$. Since every circuit of $U_{n}^r$ is the same cardinality this can be done by considering the elements of $E$ which are in parts of $\pi$ prior to $\sigma_1^M$ and the elements which are in $\sigma_1^M$. We let $I$ be the set of elements that occur prior to $\sigma_1^M$ and $J=\sigma_1^M$. The result is the indexing set above.
\end{proof}

Basic counting of the fixed points described above yields the
following.
\begin{corollary}\label{cor:group-free-uniform}
  Let $U_n^r$ be the uniform matroid of rank $r$ on $n$ elements.  We
  have the following cancellation-free, grouping-free formula for the
  antipode:
  \[
    S(U_n^r)
    =\sum_{i=0}^{r-2}
    \sum_{j=r-i+1}^{n-i}
    (-1)^{n-j+1}\binom{n}{i}\binom{n-i}{j}(U_i^i\oplus U_j^{r-i})
    +(-1)^n\sum_{x=r}^n\binom{x-1}{r-1}U_r^r.
  \]
\end{corollary}

\subsection{Matroids over hyperfields}
\label{sec:hyperfield}
In this section, we show that the method of split-merge can also be
employed to obtain a cancellation-free antipode formula for Hopf
algebras defined in \cite{ejs17} in the case of matroids over
hyperfields. To this end, we slightly change the definitions in the
previous sections. We refer the reader to \cite{ejs17} for details on
matroids over hyperfields and Hopf algebras constructed from them.

Let $E$ be a finite set, $H$ a hyperfield, and $M$ a (weak or strong) matroid over $H$ on $E$. Then, as in the matroid case, the antipode $S(M)$ of $M$ can be indexed by ordered partitions of $E$ as follows:
\[
S(M) = \sum_{k \ge 0} (-1)^k \sum_{(\pi_1,\ldots,\pi_k)\vDash E} M|_{\pi_1} \oplus (M/\pi_1)|_{\pi_2} \oplus \cdots \oplus (M/\bigcup_{i=1}^{k-2}\pi_i)|_{\pi_{k-1}} \oplus M/\bigcup_{i=1}^{k-1} \pi_i.
\]
The $\ell$-th part of an ordered partition $\pi=(\pi_1,\ldots,\pi_k)$ is defined as follows (cf.~\eqref{equation: l-part}):
\[
\ell(M,\pi) := \min\left\{t\ |\ \bigcup_{1 \le i \le t} \pi_i \text{ contains the support of a circuit of $M$}\right\}.
\]
With this, $\delta_i^M$ and $\sigma_i^M$ are defined in the same way as for ordinary matroids. We thus obtain the following.

\begin{corollary}\label{proposition: hyperfield signless}
The signless Takeuchi term $T(\pi)$ is completely determined by the $\delta$-sets of $(M,\pi)$.  Specifically,
	\[
	T(\pi) = \bigoplus_{i=1}^{j} M |_{\delta_i^M - \sigma_i^M} \oplus (M |_{\delta_i^M}) / (\delta_i^M - \sigma_i^M).
	\]
\end{corollary}

The proof is essentially the same as in Proposition \ref{prop:signless
  tak}. In fact, one can use \cite[Corollary 3.10]{ejs17} to rearrange
a partition so that it only depends on $\delta$-sets. What remains is
identical to the proof of Proposition \ref{prop:signless tak}.

Let $M$ be a matroid over a hyperfield $H$. Let $E$ be the ground set of $M$. We fix a total order $<$ on $E$ and define the involution $\iota_<$ as in Section~\ref{sec:inv}. We note that $\iota_<$ only depends on the order $<$, delta sets $\delta_i^M$, and sigma sets $\sigma_i^M$. Hence the same description as in Section~\ref{sec:inv} can be used for the hyperfield case. In particular, since $T(\pi)$ in Corollary~\ref{proposition: hyperfield signless} only depends on delta and sigma sets, one has $T(\pi)=T(\iota_<(\pi))$.

\begin{corollary}\label{coro:hyperfield theorem}
  Let $M$ be a matroid over a hyperfield $H$ with ground set $E$. Choose a total ordering on $E$, call it $<$. Then
  \[
    S(M)=\sum_{ \text{fixed points $\pi$ of } \iota_<} \ \sgn(\pi) \
    \bigoplus_{i=1}^{j} M |_{\delta_i^M - \sigma_i^M} \oplus (M
    |_{\delta_i^M}) / (\delta_i^M - \sigma_i^M).
  \]
\end{corollary}

Finally, we note that the characterization of fixed points of $\iota$ only depends on $\delta$- and $\sigma$-sets.  Hence one has the same characterization for the case of hyperfields. 

\section{Small example of computation}

We compute the antipode using the method described above for the
matroid $M$ on ground set $E = \{1,2,3,4\}$ given by circuits $\{123,124,34\}$.  Note
that $M=M[\Gamma]$ is the cycle matroid of the following graph:
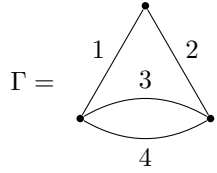
\[
  \begin{tikzpicture}[scale=.5]
    \tikzstyle{every node}=[circle,fill=black,inner sep=1pt]
    \draw
    (90:2)  node (a){}
    (210:2) node (b){}
    (330:2) node (c){}
    ;
    \draw
    (a) edge node[fill=none,label={150:$1$}] {} (b)
    edge node[fill=none,label={30:$2$}] {} (c)
    (b) edge[out=-30,in=210] node[fill=none,label={-90:$4$}] {} (c)
    (b) edge[out=30,in=150]  node[fill=none,label={90:$3$}]  {} (c)
    ;
    \draw (180:3) node[fill=none,rectangle] {$\Gamma=$};
  \end{tikzpicture}
\]

Recall that the signless Takeuchi term associated to an ordered set partition
$\pi=(\pi_1,\pi_2,\ldots,\pi_r)$ of $E$ is given by:
\[
  T(\pi)=M|_{\pi_1}
  \oplus (M/\pi_1)|_{\pi_2}
  \oplus (M/(\pi_1\cup\pi_2))|_{\pi_3}
  \oplus \cdots
  \oplus (M/(\pi_1\cup\pi_2\cup\cdots\cup\pi_{r-1}))|_{\pi_r}.
\]

When one na\"ively computes the Takeuchi formula on this matroid $M$,
one must compute the Takeuchi terms of the $75$ ordered partitions of
$\{1,2,3,4\}$, group these by equality (i.e.\ remembering the labels
associated to terms), and finally apply cancellations.  After doing so
for $M$, we have the antipode $S(M)$ given below:

\begin{align*}
  S(M)
  & =-T(1234)
  \\
  & +T(24|13)+T(23|14)+T(12|34)+T(14|23)+T(13|24)
  \\
  & -T(4|123)-T(3|124)-T(2|13|4)-T(1|23|4)-T(1|24|3)-T(2|14|3)
  \\
  & -T(2|34|1)-T(1|34|2)
  \\
  & +T(123|4)+T(124|3)
  \\
  & +T(34|12)
  \\
  & +T(2|134)+T(1|234).
\end{align*}
This yields the following grouping-free formula after identifying
isomorphism classes (the chosen labels are lexicographically minimal
in their isomorphism classes):
\[
  S(M)
  =
  -T(1234)
  +5\cdot T(12|34)
  -8\cdot T(1|23|4)
  +2\cdot T(123|4)
  +T(34|12)
  +2\cdot T(1|234).
\]

We now apply the involution method to compute $S(M)$.  First we sort
the ordered partitions of $\{1,2,3,4\}$ by their $\delta$-sets and
indicate which ``merge'' below:
\begin{align*}
  \delta^M_1 &= 34:
  &&
     \begin{matrix}
       4|3|2|1 & \merge & 34|2|1 \\
       4|3|1|2 & \merge & 34|1|2 \\
       4|3|12 & \merge & 34|12
     \end{matrix}
  \\\hline
  \delta^M_1 &= 123:
  &&
     \begin{matrix}
       3|2|1|4 & \merge & 23|1|4 \\
       2|1|3|4 & \merge & 12|3|4 \\
       3|1|2|4 & \merge & 13|2|4
     \end{matrix}
  \\\hline
  \delta^M_1 &= 124:
  &&
     \begin{matrix}
       4|2|1|3 & \merge & 24|1|3 \\
       2|1|4|3 & \merge & 12|4|3 \\
       4|1|2|3 & \merge & 14|2|3
     \end{matrix}
  \\\hline
  \delta^M_1 & =134:
  &&
     \begin{matrix}
       3|1|4|2 & \merge & 13|4|2 \\
       4|1|3|2 & \merge & 14|3|2
     \end{matrix}
  \\\hline
  \delta^M_1 & =234:
  &&
     \begin{matrix}
       3|2|4|1 & \merge & 23|4|1 \\
       4|2|3|1 & \merge & 24|3|1
     \end{matrix}
  \\\hline
  \delta^M_1 & =1234:
  &&
     \begin{matrix}
       4|2|13 & \merge & 24|13 \\
       4|1|23 & \merge & 14|23 \\
       3|2|14 & \merge & 23|14 \\
       3|1|24 & \merge & 13|24 \\
       2|1|34 & \merge & 12|34
     \end{matrix}
\end{align*}
Notice that the split-merge operation (indicated above only as merges)
has associated $18$ pairs of these ordered partitions with one
another.  This accounts for $36$ of the ordered partitions.  Thus the
fixed points of $\iota_1$ are not yet in bijection with the
cancellation-free terms of our formula.  We proceed with the second
iteration of our rule, noting that $\delta_2^M=1234$ for all ordered
partitions $\pi$ of $1234$.  Our first example below demonstrates the
importance of sorting into $\delta$-piles:
\begin{align*}
  \delta^M_1 &= 34:
  &&
     \begin{matrix}
       3|4|2|1 & \merge & 3|4|12
     \end{matrix}
  \\\hline
  \delta^M_1 &= 123:
  &&
     \begin{matrix}
       2|3|1|4 & \merge & 2|13|4 \\
       1|3|2|4 & \merge & 1|23|4 \\
       3|12|4 & \merge & 123|4
     \end{matrix}
  \\\hline
  \delta^M_1 &= 124:
  &&
     \begin{matrix}
       2|4|1|3 & \merge & 2|14|3 \\
       1|4|2|3 & \merge & 1|24|3
     \end{matrix}
  \\\hline
  \delta^M_1 &= 134:
  &&
     \begin{matrix}
       1|4|3|2 & \merge & 1|34|2 \\
       4|13|2 & \merge & 134|2
     \end{matrix}
  \\\hline
  \delta^M_1 &= 234:
  &&
     \begin{matrix}
       2|4|3|1 & \merge & 2|34|1 \\
       4|23|1 & \merge & 234|1 
     \end{matrix}
\end{align*}

Note that the involution $\iota_2$ cannot be extended (all
$\delta^M_2=1234$).  Moreover, the fixed points of $\iota_2$ yield the
cancellation-free antipode formula described above.

\begin{remark}
  Applying the above argument to the oriented matroid with
  signed circuits
  \[
    \cC=\{\pm(+,+,+,0),\pm(+,+,0,-),\pm(0,0,+,+)\}
  \]
  yields much the same computation.  In particular, one obtains the
  same cancellation-free formula for $S(M)$; it is the grouping-free
  formula that changes.  A difference here is that the terms
  $T(123|4)$ and $T(124|3)$ no longer represent the same isomorphism
  class of oriented matroids.  Thus the matroid-minor Hopf algebra
  of oriented matroids has a refined antipode.
\end{remark}

\bibliography{refs}{}
\bibliographystyle{alpha}

\end{document}